\documentclass[11pt]{amsart}

\usepackage{amssymb}
\usepackage{amsfonts}
\usepackage{latexsym}
\usepackage[all]{xy}
\usepackage{amscd}
\usepackage{comment}

\newtheorem{theorem}{Theorem}[section]
\newtheorem{lemma}[theorem]{Lemma}
\newtheorem{proposition}[theorem]{Proposition}

\theoremstyle{definition}
\newtheorem{definition}[theorem]{Definition}

\newtheorem{remark}[theorem]{Remark}
\newtheorem{note}[theorem]{Note}

\newcommand{\ot}{\otimes}
\newcommand{\C}{\mathcal{C}}
\newcommand{\D}{\mathcal{D}}
\newcommand{\M}{\mathcal{M}}
\newcommand{\X}{\mathcal{X}}
\newcommand{\G}{\mathcal{G}}

\newcommand{\K}{\mathbb{K}}
\newcommand{\lexp}[2]{{\vphantom{#2}}^{#1}{#2}}
\newcommand{\lp}{\left(}
\newcommand{\rp}{\right)}
\newcommand{\lb}{\left\{}
\newcommand{\rb}{\right\}}
\newcommand{\rt}{\triangleright}

\def\1{\hbox{\rm\rlap {1}\hskip .03in{\rm I}}}

\DeclareMathOperator{\Ker}{Ker}

\DeclareMathOperator{\Aut}{Aut}

\DeclareMathOperator{\Mod}{Mod}
\DeclareMathOperator{\id}{id}
\DeclareMathOperator{\FPdim}{FPdim}
\DeclareMathOperator{\Stab}{Stab}
\DeclareMathOperator{\degree}{degree}
\DeclareMathOperator{\Irr}{Irr}

\renewcommand{\Vec}{\operatorname{Vec}}

\renewcommand{\H}{\operatorname{H}}
\renewcommand{\dim}{\operatorname{dim \,}}

\begin{document}

\title{Crossed pointed categories and their equivariantizations}
\date{June 01, 2009} 

\author{Deepak Naidu}
\address{Department of Mathematics, Texas A\&M University,
College Station, TX 77843, USA.}
\email{dnaidu@math.tamu.edu}

\begin{abstract}
We propose the notion of {\em quasi-abelian third cohomology} of
crossed modules, generalizing Eilenberg and MacLane's abelian cohomology and 
Ospel's quasi-abelian cohomology, and classify crossed pointed categories in terms of it.
We apply the process of equivariantization to the latter to obtain
braided fusion categories which may be viewed as generalizations of
the categories of modules over twisted Drinfeld doubles of finite 
groups. As a consequence, we obtain a description of {\em all} braided group-theoretical
categories. A criterion for these categories to be modular is given.
We also describe the quasi-triangular quasi-Hopf algebras underlying
these categories.
\end{abstract}

\maketitle

\begin{section}
{Introduction}

The notion of a crossed category (short for braided group-crossed category), 
introduced by Turaev \cite{Tu1,Tu2},
has attracted much attention recently \cite{DGNO,Ki1,Ki2,M1,M2}.
Roughly, a crossed category consists of a group $G$, a $G$-graded tensor
category $\C$, an action $g \mapsto T_g$ of $G$ on $\C$ by tensor
autoequivalences, and $G$-braidings
$c(X,Y): X \ot Y \xrightarrow{\sim} T_g(Y) \ot X,\, X,Y \in \C$, 
satisfying certain compatibility conditions.
Crossed categories are know to arise in various contexts; for instance,
in \cite{M1}, M\"uger showed that Galois extensions of braided
tensor categories have a natural structure of crossed categories. M\"uger also
established a connection between $1$-dimensional quantum field
theories and crossed categories \cite{M2}. Furthermore, crossed categories
have been shown, by Kirillov Jr., also to arise in the theory of vertex 
operator algebras \cite{Ki2}.

Recall that a fusion category is said to be {\em pointed} if all
its simple objects are invertible. 
One of our goals in the present note is to classify all 
crossed pointed categories. 
It is known \cite{JS} that braided pointed categories are classified
by Eilenberg and MacLane's abelian cohomology $\H_{ab}^3(A,\K^\times)$,
where $A$ is a finite abelian group.
On the other hand, in \cite{Tu1,Tu2}, a description of certain crossed pointed
categories in which the group action is strict is given in terms of 
Ospel's quasi-abelian cohomology $\H_{qa}^3(G,\K^\times)$, where $G$ is
a (not necessarily abelian) finite group.  
As remarked in \cite[Subsection 4.9]{M2}, to obtain
a complete classification of crossed pointed categories one must
allow for non-strict group actions. To this end, we generalize
Ospel's quasi-abelian cohomology to define the notion of
quasi-abelian third cohomology $\H_{qa}^3(\X,\K^\times)$ of a crossed module
$\X$ (see Definition~\ref{quasi-abelian 3-cocycle}). 
To any given $\xi \in Z_{qa}^3(\X,\K^\times)$ we associate a crossed
pointed category $\C(\xi)$ and show that all 
crossed pointed categories are of this form.

Another notion that has been studied extensively recently is that of
a modular category.
Examples of modular categories arise from several diverse areas such as 
quantum group theory, $3$-dimensional topology, vertex operator algebras, 
rational conformal field theory, etc.
Let $G$ be a finite group.
Perhaps the most accessible construction of a modular category 
is that of the category of
modules over the Drinfeld double $D(G)$ of $G$. Let $\omega$ be 
a $3$-cocycle on $G$.
In \cite{DPR1, DPR2} Dijkgraaf, 
Pasquier, and Roche introduced a quasi-triangular quasi-Hopf algebra 
$D^\omega(G)$, generalizing the Drinfeld double $D(G)$.
It is well known that the 
category $D^\omega(G)$-$\Mod$ of modules over $D^\omega(G)$ is a modular category.
Modular categories resembling $D^\omega(G)$-$\Mod$ arises naturally from crossed pointed categories.
Specifically, an important feature of a general crossed fusion category is that the application
of the equivariantization process (which is analogous to taking the invariants under a groups action) 
yields a braided fusion category.
We apply the equivariantization process to the aforementioned 
crossed pointed category $\C(\xi)$ and study the resulting braided fusion category,
which resembles the category $D^\omega(G)$-$\Mod$.
As a consequence, we obtain a description of {\em all} braided group-theoretical
categories.
We show that $\C(\xi)$ is modular if and only if 
$\xi$ is nondegenerate in the sense of 
Definition~\ref{nondegenerate quasi-abelian 3-cocycle} and a certain
homomorphism is surjective (see Proposition~\ref{when nd}).

By a general result, the equivariantization of the category $\C(\xi)$ is equivalent,
as a braided fusion category, to the category of modules over
some finite-dimensional quasi-triangular quasi-Hopf algebra $H$.
In the sequel we describe such an $H$. Namely,
given $\xi \in Z_{qa}^3(\X,\K^\times)$, we construct a 
finite-dimensional quasi-triangular quasi-Hopf algebra $H(\xi)$,
generalizing $D^\omega(G)$, and show that 
$\C(\xi) \cong H(\xi)$-$\Mod$, as braided fusion categories.

The content of this note is as follows. In Section 2, we recall 
some essential definitions and results concerning nondegenerate fusion
categories, equivariantization, and crossed categories.
In Section 3, we propose the notion of quasi-abelian third cohomology
of crossed modules. Section 4 consists of a construction of 
crossed pointed categories from quasi-abelian $3$-cocycles
and a classification of the former. In Section 5, we apply the process of equivariantization
to the categories obtained in Section 4 and study the resulting braided
fusion categories. In Section 6, we construct finite-dimensional quasi-triangular
quasi-Hopf algebras from quasi-abelian $3$-cocycles, which are shown to underlie
the braided fusion categories obtained in the Section 5. 

\textbf{Acknowledgments:} We thank D. Nikshych and S. Witherspoon
for useful discussions.
\end{section}

\begin{section}
{Preliminaries}

In this note, we will freely use the language and basic theory of fusion categories
and modular categories \cite{BK, Os, ENO}. 
In what follows we recall some essential definitions and results.

\begin{subsection}
{Convention}
Let $\K$ be an algebraically closed field of characteristic $0$. 
The multiplicative group of non-zero elements of $\K$ will be denoted by $\K^\times$. 
Unless otherwise stated, all cocycles appearing in this 
work will have coefficients in the trivial module $\K^\times$.
All functors will be assumed to be additive and $\K$-linear on the morphism spaces.
The unit object of a tensor category will be denoted by $\1$. 
The identity element of a group will be denoted by $e$.
\end{subsection}

\begin{subsection}{Morita equivalence}

Following \cite{M4}, we say that two fusion categories $\C$ and $\D$ are
Morita equivalent if $\D$ is equivalent to the dual fusion category
$\C^*_\M$, for some indecomposable right $\C$-module category $\M$
(see also \cite{ENO, O}). 
The above is known to be an equivalence relation on the class of 
fusion categories. A fusion category is said to be {\em pointed} if
all it's simple objects are invertible.
A fusion category is {\em group-theoretical} 
if it is Morita equivalent to a pointed category. 
\end{subsection}

\begin{subsection}{Nondegenerate fusion categories} 

Let $\C$ be a braided fusion category with braiding $c$.
Two objects $X$ and $Y$ of $\C$ are said to {\em centralize}
each other if $c(Y,X) \circ c(X,Y)=\id_{X\ot Y}$ \cite{M3}.

For any fusion subcategory $\D \subseteq \C$ its {\em centralizer}
$\D'$ is the full fusion subcategory of $\C$ consisting of all objects $X \in \C$
which centralize every object in $\D$. The category $\C$ is said to be
{\em non-degenerate} if $\C'=\Vec$ (the fusion category generated
by the unit object). If $\C$ is a pre-modular category,
i.e., has a twist, then it is non-degenerate if and only if it is
modular \cite{BB, M3, DGNO}.

The following proposition will be used later.

\begin{proposition}
\label{twists and self-dual invertibles}
Let $\C$ be a nondegenerate fusion category.
Suppose $\C$ admits a twist. Then the set of 
twists on $\C$ is in bijection with
the set of invertible self-dual objects of $\C$.
\end{proposition}
\begin{proof}
Let $\Aut_\ot(\id_\C)$ denote the group of tensor automorphisms
of the identity tensor functor $\id_\C$. 
Define $\Aut_\ot^*(\id_\C) := \{ \varphi \in \Aut_\ot(\id_\C) \mid
\varphi_{X^*} = (\varphi_X)^*,\; \forall X \in \C\}$. 
Let $\theta$ be a fixed twist on $\C$. The map $\varphi \mapsto \theta_{\varphi},
(\theta_{\varphi})_X := \theta_X \circ \varphi_X$ for all $X \in \C$, is a bijection
from $\Aut_\ot^*(\id_\C)$ to the set of all twists on $\C$.

Let $X_1,X_2,\cdots$ denote the simple
objects of $\C$ and let $G(\C)$ denote the group of invertible objects
of $\C$. Also, let $S$ denote the $S$-matrix of $\C$ with respect to
$\theta$. It was shown in \cite{GN} that the map
$$
G(\C) \to \Aut_\ot(\id_\C): X_j \mapsto \varphi_j, \qquad 
(\varphi_j)_{X_i} := \frac{S_{ij}}{d(X_i) d(X_j)} \id_{X_i}
$$
is an isomorphism. It is easy to check that this map restricts to
a bijection between the set of invertible self-dual object of $\C$
and the set $\Aut_\ot^*(\id_\C)$.
\end{proof}

\end{subsection}

\begin{subsection}
{Equivariantization}
Recall that a tensor functor between two tensor categories
$\C$ and $\D$ is a triple $(F, \varphi, \varphi_0)$ where
$F:\C \to \D$ is a functor, $\varphi$ is a natural isomorphism
$F \circ \ot_{\C} \xrightarrow{\sim} 
\ot_{\D} \circ (F \times F)$, and $\varphi_0$
is an isomorphism $F(\1_{\C}) \xrightarrow{\sim} \1_{\D}$ satisfying
certain compatibility conditions (see \cite{K}). We will call
$\varphi$ the {\em tensor structure} on $F$ and $\varphi_0$ the {\em unit-preserving structure}
on $F$. For a group $G$, we will denote by $\underline{G}$ the tensor
category whose objects are elements of $G$, morphisms are the identities,
and whose tensor product is given by the group operation in $G$.

Let $\C$ be a fusion category with an action of a finite 
group $G$ given by a tensor functor $T:\underline{G}
\to \Aut_\ot(\C); g \mapsto T_g$. 
Let $\gamma$ be the tensor structure on the functor $T$.
In this situation one can define the notion of a $G$-equivariant object in $\C$.
Namely, a {\em $G$-equivariant object\,} in $\C$ is a pair $(X, \{u_g\}_{g\in G})$
where $X$ is an object of $\C$ and   
\begin{equation}
\label{equivariant structure}
u_g: T_g(X)\xrightarrow{\sim} X, \qquad g\in G,
\end{equation}
is a family of  isomorphisms called {\em equivariant structure on $X$} such that 
\begin{equation}
\label{equivariant structure condition}
u_{gh}= u_g \circ T_g(u_h) \circ \gamma_{g,h}(X),
\end{equation}
for all $g,h \in G$. 

One defines morphisms between equivariant objects to be morphisms in $\C$ commuting 
with the equivariant structures. The {\em equivariantization of $\C$}, denoted $\C^G$,
is the category of $G$-equivariant objects of $\C$ \cite{Ki1, AG,G,Ta}. The equivariantization category
$\C^G$ is a fusion category with tensor product given by the following.
Let $(X, \{u_g\}_{g\in G}), (X', \{u'_g\}_{g\in G}) \in \C^G$. Then
$$
(X, \{u_g\}_{g\in G}) \ot (X', \{u'_g\}_{g\in G}) := (X \ot X', \{\tilde{u}_g\}_{g\in G}),
$$
where
\begin{equation}
\label{tensor product in equivariantization}
\tilde{u}_g := (u_g \ot u'_g) \circ \mu_g(X,X')
\end{equation}
for all $g \in G$. Here $\mu_g$ is the tensor structure on the functor $T_g, g \in G$.

\begin{remark}
We have $\FPdim(\C^G) = |G| \FPdim(\C)$.
\end{remark}

\end{subsection}

\begin{subsection}
{Crossed categories}

Recall that a  {\em grading} of a fusion category $\C$ 
by a finite group $G$ is a decomposition
$$
\C =\bigoplus_{g\in G}\, \C_g
$$
of $\C$ into a direct sum of full abelian subcategories such that
$\otimes$ maps $\C_g\times \C_h$ to $\C_{gh}$ and $*$ maps $\C_g$ to $\C_{g^{-1}}$, for all $g,h\in G$. 
Note that $\C_e$, called the {\em trivial component}, is a fusion subcategory of $\C$.  
A grading is said to {\em faithful} if  $\C_g\neq 0$ for all $g\in G$.

Below we recall the notion of a {\em crossed category}
(short for {\em braided group-crossed category}),
introduced by Turaev \cite{Tu1, Tu2}, in a more general form 
(see also \cite{DGNO,M1,M2}).

\begin{definition}
\label{crossed category}
A {\em crossed fusion category} is an eight-tuple
$(\C, G, T, \gamma, \iota, \mu, \nu, c)$, where
\begin{itemize}
\item $G$ is a finite group,
\item $\C$ is a fusion category with a (not necessarily faithful)
$G$-grading $\displaystyle \C~=~\bigoplus_{g \in G}\C_g$,
\item $T: \underline{G} \to \Aut_{\ot}(\C):g \mapsto T_g$ is a tensor functor,
satisfying $T_g(\C_h) \subset \C_{ghg^{-1}}$,
with tensor structure $\gamma$ and unit-preserving structure
$\iota$,
\item $\mu$ is a family $\{\mu_g\}_{g \in G}$ where $\mu_g$
is a tensor structure on $T_g$,
\item $\nu$ is a family $\{\nu_g\}_{g \in G}$ where $\nu_g$
is a unit-preserving structure on $T_g$,
\item $c(X,Y) : X \ot Y \xrightarrow{\sim} T_g(Y) \ot X,
\,\, X \in \C_g,Y \in \C$, is a family of natural isomorphisms, called {\em $G$-braiding},
\end{itemize}
satisfying the following compatibility conditions:
\begin{equation}
\label{T preserves c}
\tag{i}
\begin{split}
(\gamma_{g,h}(Y) \ot \id_{T_g(X)}) \circ &
(\gamma^{-1}_{ghg^{-1},g}(Y) \ot  \id_{T_g(X)}) \circ
c(T_g(X),T_g(Y)) \circ \mu_g(X,Y) \\
&= \mu_g(T_{h}(Y),X) \circ T_g(c(X,Y)), 
\end{split}
\end{equation}
for all $g,h \in G$ and objects $X \in \C_h ,Y \in \C$.
\begin{equation}
\label{c1}
\tag{ii}
\begin{split}
\alpha^{-1}_{T_g(T_h(Z)),X,Y} \circ &
(\gamma_{g,h}(Z) \ot \id_{X \ot Y}) \circ
c(X \ot Y,Z) \circ \alpha^{-1}_{X,Y,Z} \\
&= (c(X,T_h(Z)) \ot \id_Y) \circ
\alpha^{-1}_{X,T_h(Z),Y} \circ
(\id_X \ot c(Y,Z)), 
\end{split}
\end{equation}
for all $g,h \in G$ and objects $X \in \C_g,Y \in \C_h,Z \in \C$.
\begin{equation}
\label{c2}
\tag{iii}
\begin{split}
\alpha_{T_g(Y),T_g(Z),X} \circ &
(\mu_g(Y,Z) \ot \id_X) \circ
c(X,Y \ot Z) \circ \alpha_{X,Y,Z} \\
& =(\id_{T_g(Y)} \ot c(X,Z)) \circ
\alpha_{T_g(Y),X,Z} \circ
(c(X,Y) \ot \id_Z), 
\end{split}
\end{equation}
for all $g \in G$ and objects $X \in \C_g,Y,Z \in \C$.

(Here $\alpha$ denotes the associativity constraint of $\C$.)
\end{definition}

\begin{remark}
The trivial component of a crossed fusion category is a braided fusion category.
\end{remark}

Now let $\C:=(\C, G, T, \gamma, \iota, \mu, \nu, c)$
be a crossed fusion category.
It is explained in \cite{Ki1} and \cite{M1} that 
the equivariantization category $\C^G$ admits a braiding, 
i.e, $C^G$ is a braided fusion category.
The braiding $\tilde{c}$ on $\C^G$ is defined as follows.
Let $(X,\, \{u_g\}_{g\in G})$ and $(X',\, \{u'_g\}_{g\in G})$ be 
objects of  $\C^G$.
Let $X =\oplus_{g\in G}\,X_g$ be a decomposition of $X$ with respect
to the $G$-grading of $\C$.  Then $\tilde{c}_{X,X'}$ is given by the composition
\begin{equation}
\label{braiding in equivariantization}
X\ot X' =\bigoplus_{g\in G}\, X_g\ot X' \xrightarrow{\oplus\,c_{X_g,X'}}
\bigoplus_{g\in G}\, T_g(X') \ot X_g \xrightarrow{\oplus\, u'_g \ot \id_{X_g}} 
\bigoplus_{g\in G}\, X' \ot X_g = X'\ot X.
\end{equation}

\begin{remark}
\label{when C^G is nd}
It is shown in \cite{DGNO} that the equivariantization category $\C^G$ 
is nondegenerate if and only if the $G$-grading is
faithful and the trivial component $\C_e$ is nondegenerate.
\end{remark}

\begin{definition}
\label{crossed functor}
Let $\C=(\C, G, T, \gamma, \iota, \mu, \nu, c)$ and
$\C'=(\C', G', T', \gamma', \iota', \mu', \nu', c')$
be crossed fusion categories. A {\em crossed tensor
functor} from $\C$ to $\C'$ is a $5$-tuple $(f,F,\eta,\eta_0,\beta)$
where
\begin{itemize}
\item $f:G \to G'$ is a group homomorphism,
\item $F:\C \to \C'$ is a tensor functor with tensor structure
$\eta$ and unit-preserving structure $\eta_0$,
\item $\beta$ is a family $\{\beta_g\}_{g \in G}$ where 
$\beta_g:F \circ T_g \xrightarrow{\sim} T'_{f(g)} \circ F$ is
an isomorphism of tensor functors
\end{itemize}
satisfying the following compatibility conditions:
\begin{equation}
\label{F preserves grading}
\tag{i}
F(\C_g) \subseteq \C'_{f(g)}, 
\end{equation}
for all $g \in G$.
\begin{equation}
\label{F preserves c}
\tag{ii}
(\beta_g(Y) \ot \id_{F(X)}) \circ
\eta(T_g(Y),X) \circ F(c(X,Y))
= c'(F(X),F(Y)) \circ \eta(X,Y), 
\end{equation}
for all $g \in G$ and objects $X \in \C_g,Y \in \C$.
\begin{equation}
\label{beta and gamma are compatible}
\tag{iii}
T'_{f(g)}(\beta_h(X)) \circ \beta_g(T_h(X)) \circ F(\gamma_{g,h}(X)) 
= \gamma'_{f(g),f(h)}(F(X)) \circ \beta_{gh}(X), 
\end{equation}
for all $g,h \in G$ and objects $X \in \C$.

We say that $(f,F,\eta,\eta_0,\beta)$ is an {\em equivalence} if $f$ is an
isomorphism and $F$ is an equivalence.
\end{definition}

\end{subsection}
\begin{subsection}
{Pointed Categories}
\label{pointed categories}
Recall that a fusion category is said to be {\em pointed} if
all its simple object are invertible.

Let $X$ be a finite group and $\omega$ be a $3$-cocycle on $X$. We associate
to the pair $(X,\omega)$ a pointed category $\Vec_X^\omega$ as follows.
The objects of $\Vec_X^\omega$ are $X$-graded finite dimensional vector
spaces over $\K$ and morphisms are linear transformations that respect the
grading. The unit object of $\Vec_X^\omega$ is ground field $\K$ supported
on $\{e\}$. The tensor product $V \ot W$ of homogeneous objects $V,W \in \Vec_X^\omega$
of degrees $x,y \in X$, respectively, is defined to be the homogeneous object
$V \ot_\K W$ of degree $xy$.

The associativity constraint $\alpha$ is defined by
$$
\alpha_{U,V,W}:
(U \ot V) \ot W \xrightarrow{\sim} U \ot (V \ot W): 
(u \otimes v) \otimes w 
\mapsto \omega(x,y,z) u \otimes (v \otimes w),
$$
where $U,V,W \in \Vec_G^\omega$ and $u \in U, v \in V, w \in W$
are homogeneous elements of degrees $x,y,z \in X$, respectively.

The left and right unit constraints $\lambda$ and $\rho$,
respectively, are defined by
$$
\lambda_{V} := \K \ot V \xrightarrow{\sim} V:
1 \ot v \mapsto \omega(e,e,x)^{-1} v,
$$
and
$$
\rho_{V} := V \ot \K \xrightarrow{\sim} V:
v \ot 1 \mapsto \omega(x,e,e) v.
$$ 
where $V \in \Vec_G^\omega$ and $v \in V$
is a homogeneous element of degree $x \in X$.

Every pointed category is equivalent to some $\Vec_X^\omega$.
\end{subsection}
\begin{subsection}
{Crossed Modules}
Recall that a {\em (finite) crossed module} is a triple $(G,\, X,\, \partial)$,
where $G$ and $X$ are (finite) groups with $G$ acting on $X$ as automorphisms,
denoted $(g, x) \mapsto \lexp{g}{x}$, and $\partial : X \to G$ 
is a homomorphism satisfying
$$ 
\lexp{\partial(x)}{x'} = xx'x^{-1}, \qquad x,x' \in X,
$$
and
$$ 
\partial(\lexp{g}{x}) = g \partial(x) g^{-1}, \qquad g \in G, x \in X.
$$

Note that $\Ker \partial$ is a central subgroup of $X$.

A homomorphism of crossed modules $(G,\, X,\, \partial) \to (G',\, X',\, \partial')$
is a pair of group homomorphisms $(f: G\to G', \, F: X \to X')$ 
such that $\partial'\circ F = f \circ \partial$ and 
$F(\lexp{g}{x}) = \lexp{f(g)}{F(x)}$, $g \in G$. We say that $(f,F)$ is
an {\em isomorphism} if both $f$ and $F$ are isomorphisms.

\end{subsection}

\end{section}

\begin{section}
{Quasi-abelian third cohomology of crossed modules}
\label{quasi-abelian cohomology}

Let $A$ be an abelian group. Eilenberg and Mac Lane \cite{EM1,EM2,ML}
argue that the cohomology groups $\H^n(A,\K^\times)$ are inappropriate 
since then do not take into account the abelianess of $A$, so should be 
replaced by groups $\H_{ab}^n(A,\K^\times)$. (For the cohomology theory for crossed modules, see \cite{W}.)
Below we recall the definition of $\H_{ab}^3(A,\K^\times)$.

An {\em abelian $3$-cocycle} on $A$ is a pair $(\omega, c)$, where
$\omega$ is a normalized $3$-cocycle on $A$, i.e.,
\begin{eqnarray*}
\omega(x,y,z) &=& 1, \text{ if } x,y, \text{ or } z \text{ is identity},\\
\omega(x,y,z) \omega(w,xy,z) \omega(w,x,y) &=& \omega(w,x,yz) \omega(wx,y,z),
\end{eqnarray*}
for all $w,x,y,z \in A$,
and $c$ is a $2$-cochain on $A$ (i.e., $c \in C^2(A, \K^\times)$) satisfying the following
equations:
\begin{eqnarray*}
c(xy,z) &=& \frac{\omega(x,y,z) \omega(z,x,y)}
{\omega(x,z,y)}
c(x,z) c(y,z), \\
c(x,yz) &=& \frac{\omega(y,x,z)}
{\omega(x,y,z) \omega(y,z,x)}
c(x,y) c(x,z)
\end{eqnarray*}
for all $x,y,z \in A$.

Abelian $3$-cocycles on $A$ form an abelian
group, denoted by $Z_{ab}^3(A,\K^\times)$, under pointwise
multiplication. The group of coboundaries is defined by
$$
B^3_{ab}(A,\K^\times) :=
\lb \lp d\eta, \, 
\, (x,y) \mapsto \frac{\eta(y,x)}{\eta(x,y)} \rp \, \vline \,
\text{ normalized } \eta \in C^2(G,\K^\times) \rb,
$$
which is a subgroup of $Z^3_{ab}(A,\K^\times)$.
The quotient $Z^3_{ab}(A,\K^\times)/B^3_{ab}(A,\K^\times)$ is
the {\em abelian third cohomology} of $A$ denoted $\H^3_{ab}(A,\K^\times)$.

\begin{remark}
The group $\H^3_{ab}(A,\K^\times)$ is isomorphic to the group of quadratic forms on $A$, see \cite{ML}.
\end{remark}

\begin{definition}
An abelian $3$-cocycle $(\omega,c)$ on $A$ is {\em nondegenerate} if the symmetric
bicharacter
$$
A \times A \to \K^\times: (x,y) \mapsto c(y,x) c(x,y)
$$
is nondegenerate.
\end{definition}

In \cite{O}, C. Ospel generalized the notion of abelian third cohomology in
the following way. Let $G$ be a (not necessarily abelian) group.
A  {\em quasi-abelian $3$-cocycle} on $G$
is a pair $(\omega,c)$, where $\omega$ is a $3$-cocycle on $G$ and $c$ is
a $2$-cochain on $G$ (i.e., $c \in C^2(G, \K^\times)$) satisfying the following
equations:
\begin{eqnarray*}
\omega(gxg^{-1},gyg^{-1},gzg^{-1}) &=& \omega(x,y,z),\\
c(gxg^{-1},gyg^{-1}) &=& c(x,y), \\
c(xy,z) &=& \frac{\omega(x,y,z) \omega(xyz(xy)^{-1},x,y)}
{\omega(x,yzy^{-1},y)}
c(x,yzy^{-1}) c(y,z), \\
c(x,yz) &=& \frac{\omega(xyx^{-1},x,z)}
{\omega(x,y,z) \omega(y,z,x)}
c(x,y) c(x,z),
\end{eqnarray*}
for all $g,x,y,z \in G$.

\begin{note}
The third equation above appeared
in a slightly different but equivalent form in \cite{O}.
\end{note}

Quasi-abelian $3$-cocycles on $G$ form an abelian
group, denoted by $Z_{qa}^3(G,\K^\times)$, under pointwise
multiplication. The group of coboundaries is defined by
$$
B^3_{qa}(G,\K^\times) :=
\lb \lp d(\eta), \, 
\, (x,y) \mapsto \frac{\eta(y,x)}{\eta(x,y)} \rp \, \vline \,
\text{ conjugation-invariant } \eta \in C^2(G,\K^\times) \rb,
$$
which is a subgroup of $Z^3_{qa}(G,\K^\times)$.
The quotient $Z^3_{qa}(G,\K^\times)/B^3_{qa}(G,\K^\times)$ is
the {\em quasi-abelian third cohomology} of $G$ denoted $\H^3_{qa}(G,\K^\times)$.
When $G$ is abelian, quasi-abelian cohomology reduces to abelian cohomology:
$\H^3_{ab}(G,\K^\times) = \H^3_{qa}(G,\K^\times)$.

We extend Ospel's quasi-abelian cohomology for groups to cover crossed modules, as follows.
We allow $G$ to act on an arbitrary group $X$ (not just $X=G$). The first condition $\omega^g=\omega$
in Ospel's definition is replaced by the condition `$\omega^g$ is cohomologous to $\omega$ via $\mu_g$'.
The second condition $c^g = c$ is extended similarly, as are the other conditions. This results in
the following definition, whose main motivation is the classification of crossed pointed categories
(see Section~4).

\begin{definition}
\label{quasi-abelian 3-cocycle}
A {\em quasi-abelian $3$-cocycle} on a crossed module
$\X = (G,X,\partial)$ is a quadruple $(\omega, \gamma, \mu, c)$
where
\begin{equation}
\label{condition on omega}
\tag{a}
\omega \in Z^3(X,\K^\times),
\end{equation}
\begin{equation}
\label{condition on gamma}
\tag{b}
\gamma \in Z^2(G, C^1(X,\K^\times)),
\end{equation}
$\mu \in C^1(G,C^2(X,\K^\times))$ satisfying
\begin{equation}
\label{relation between mu and omega}
\tag{c}
d(\mu_g) = \frac{\omega^g}{\omega}, \qquad g \in G,
\end{equation}
that is,
\begin{equation*}
\frac{\mu_g(y,z)\mu_g(x,yz)}{\mu_g(xy,z)\mu_g(x,y)} = \frac{\omega^g(x,y,z)}{\omega(x,y,z)}, \qquad g \in G,  x,y,z \in X,
\end{equation*}
\begin{equation}
\label{relation between gamma and mu}
\tag{d}
d(\gamma_{g,h}) = (d\mu)_{g,h}, \qquad g,h \in G,
\end{equation}
that is,
\begin{equation*}
\frac{\gamma_{g,h}(x)\gamma_{g,h}(y)}{\gamma_{g,h}(xy)} = \frac{\mu_g(\lexp{h}{x}, \lexp{h}{y})\mu_h(x,y)}{\mu_{gh}(x,y)}, \qquad g,h \in G, x,y \in X
\end{equation*}
and $c \in C^2(X,\K^\times)$ satisfying
\begin{equation}
\label{relation between c, mu, and gamma}
\tag{e}
\frac{c^g(x,y)}{c(x,y)} = \frac{\mu_g(xyx^{-1},x)}{\mu_g(x,y)}
\frac{\gamma_{g\partial(x)g^{-1},g}(y)}{\gamma_{g, \partial(x)}(y)},
\qquad g \in G, x,y \in X,
\end{equation}
\begin{equation}
\label{relation between c, omega, and gamma}
\tag{f}
c(xy,z) = \frac{\omega(x,y,z) \omega((xy)z(xy)^{-1},x,y)}
{\omega(x,yzy^{-1},y) \gamma_{\partial(x),\partial(y)}(z)}
c(x,yzy^{-1}) c(y,z), \qquad x,y,z \in X,
\end{equation}
and
\begin{equation}
\label{relation between c, omega, and mu}
\tag{g}
c(x,yz) = \frac{\omega(xyx^{-1},x,z)}
{\omega(x,y,z) \omega(xyx^{-1},xzx^{-1},x)  \mu_{\partial(x)}(y,z)}
c(x,y) c(x,z), \qquad x,y,z \in X.
\end{equation}
\end{definition}

\begin{note}
\label{notation}
Some remarks about the notation: 
$C^n$ denotes the space of $n$-cochains,
$Z^n$ denotes the space of $n$-cocycles,
and $d$ is the usual {\em differential operator} \cite{B}. (We note
that the definition of $d$ depends on whether the module
under consideration is left or right.)  The action of $G$ on $X$,
$(g,x) \mapsto \lexp{g}{x}$,
induces a right action of $G$ on $C^n(X,\K^\times)$ 
by translations. 
The map $c^g \in C^2(X,\K^\times)$ is defined by
$c^g(x,y):=c(\lexp{g}{x},\lexp{g}{y})$ and the map $\omega^g$ is defined
similarly. 
\end{note}

Quasi-abelian $3$-cocycles on a crossed module $\X=(G,X,\partial)$ form an abelian
group, denoted by $Z_{qa}^3(\X,\K^\times)$, under pointwise
multiplication.

We define the group of coboundaries by
$$
B^3_{qa}(\X,\K^\times) :=
\lb \lp d\eta, \, d\beta, \, g \mapsto d(\beta_g) \frac{\eta^g}{\eta},
\, (x,y) \mapsto \beta_{\partial(x)}(y) 
\frac{\eta(xyx^{-1},x)}{\eta(x,y)} \rp \vline
\mathop{}^{\eta \in C^2(X,\K^\times),}_{\beta \in C^1(G,C^1(X,\K^\times))} \rb.
$$
A direct computation shows that $B^3_{qa}(\X,\K^\times) 
\subseteq Z^3_{qa}(\X,\K^\times)$.
\begin{definition}
The {\em quasi-abelian third cohomology} of a crossed
module $\X$, denoted $\H_{qa}^3(\X,\K^\times)$, is the
quotient of $Z_{qa}^3(\X,\K^\times)$ by $B_{qa}^3(\X,\K^\times)$.
\end{definition}

\begin{remark}
Let $G$ be a group. Consider the crossed module
$\G = (G,G,\id_G)$, where $G$ acts on itself by conjugation.
\begin{enumerate}
\item[(i)] There is a homomorphism $\H^3_{qa}(G,\K^\times) \to \H^3_{qa}(\G,\K^\times)$
induced from
$$
Z^3_{qa}(G,\K^\times) \to Z^3_{qa}(\G,\K^\times):
(\omega,c) \mapsto (\omega,1,1,c). 
$$
\item[(ii)] There also exists a homomorphism
$\H^3(G,\K^\times) \to \H^3_{qa}(\G,\K^\times)$
(see Lemma \ref{3-cocycle to quasi-abelian 3-cocycle}).
\end{enumerate}
\end{remark}

\begin{definition}
A quasi-abelian $3$-cocycle $(\omega, \gamma, \mu, c)$ is normalized if
\begin{eqnarray*}
\omega(x,y,z) = 1, \text{ if } x,y, \text{ or } z \text{ is identity}, & &
\gamma_{g,h}(x) = 1, \text{ if } g,h, \text{ or } x \text{ is identity},\\
\mu_g(x,y) = 1, \text{ if } x,y, \text{ or } g \text{ is identity}, & &
c(x,y) = 1, \text{ if } x \text{ or } y \text{ is identity}.
\end{eqnarray*}
\end{definition}

\begin{note}
Every quasi-abelian $3$-cocycle is cohomologous to a normalized one.
\end{note}

Let $(\omega, \gamma, \mu, c)$ be a normalized quasi-abelian $3$-cocycle on
a crossed module $(G,X,\partial)$. Then
$(\omega|_{\Ker \partial}, c|_{\Ker \partial})$ is an abelian $3$-cocycle
on the (abelian) group $\Ker \partial$.

\begin{definition}
\label{nondegenerate quasi-abelian 3-cocycle}
A normalized quasi-abelian $3$-cocycle $(\omega, \gamma, \mu, c)$ on
a crossed module $(G,X,\partial)$ is {\em nondegenerate} if the
abelian $3$-cocycle $(\omega|_{\Ker \partial}, c|_{\Ker \partial})$ 
on the (abelian) group $\Ker \partial$ is nondegenerate.
\end{definition}

Any homomorphism $(f,F):(G',X',\partial')= \X' \to \X=(G,X,\partial)$ of
crossed modules induces a homomorphism
$$
Z^3_{qa}(\X,\K^\times) \to Z^3_{qa}(\X',\K^\times):
(\omega,\gamma,\mu,c) \mapsto (\omega,\gamma,\mu,c)^{(f,F)}
$$
where 
$$
(\omega,\gamma,\, \mu,c)^{(f,F)} = 
\lp\omega \circ F^{\times 3}, \, (g,h) \mapsto \gamma_{f(g),f(h)} \circ F, 
\, g \mapsto \mu_{f(g)} \circ F^{\times 2}, \, c \circ F^{\times 2} \rp.
$$
It is straight-forward to verify that the above homomorphism preserves coboundaries,
thereby it provides a homomorphism $\H^3_{qa}(\X,\K^\times) \to \H^3_{qa}(\X',\K^\times)$.
Consequently, for any crossed module $\X$ there is a natural action of the 
group of automorphisms $\Aut(\X)$ of $\X$ on $\H^3_{qa}(\X,\K^\times)$.

\end{section}

\begin{section}
{Classification of crossed pointed categories}

In this section we classify crossed pointed categories
in terms of quasi-abelian third cohomology of crossed modules. 

\begin{subsection}
{Construction of a crossed pointed category from a quasi-abelian
$3$-cocyle on a crossed module}
\label{construction}

Given a quasi-abelian $3$-cocycle $(\omega, \gamma, \mu, c)$ on a
finite crossed module $(G,X,\partial)$, we associate to it a 
crossed pointed category 
$(\C, G, T, \tilde{\gamma}, \iota, \tilde{\mu}, \nu, \tilde{c})$
as follows. As a fusion category $\C = \Vec_X^\omega$. For each
$g \in G$, let $\C_g$ denote the full abelian subcategory consisting of
objects of $\Vec_X^\omega$ supported on $\partial^{-1}(g) \subset X$, i.e., 
objects of $\C_g$ are defined to be finite-dimensional $\partial^{-1}(g)$-graded vector spaces
(we set $\C_g := \{0\}$ if $\partial^{-1}(g)$ is empty).
This defines a
$G$-grading of $\C$: $\displaystyle \C = \bigoplus_{g \in G} \C_g$.

Next we define a functor $T: \underline{G} \to \Aut_\ot(\C)
: g \mapsto T_g$ as follows. 
Let $V \in \Vec_X^\omega$ be a homogeneous object of degree $x \in X$.
The functor 
$T_g: \Vec_X^\omega \xrightarrow{\sim} \Vec_X^\omega$
is defined by $T_g(V):= V$ (as a vector space) and 
the degree of $T_g(V)$ is defined to be $\lexp{g}{x}$. The $T_g$'s are extended
to nonhomogeneous objects and morphisms in the obvious way.

The tensor structure $\tilde{\gamma}$ on the functor 
$T: \underline{G} \to \Aut_\ot(\C)$ is defined by
$$
\gamma_{g,h}(x) \id_V =: \tilde{\gamma}_{g, h}(V)
:T_{gh}(V) \xrightarrow{\sim} (T_g \circ T_h)(V), 
$$
for all homogeneous objects $V \in \Vec_X^\omega$ of degree $x\in X,$ and $g,h \in G$.

The unit-preserving structure $\iota: T_e \xrightarrow{\sim} \id_{\C}$ on
the functor $T: \underline{G} \to \Aut_\ot(\C)$ is defined by
$$
\gamma^{-1}_{e,e}(x) \id_V =: \iota(V): T_{e}(V) \xrightarrow{\sim} 
\id_{\C}(V),
$$
for all homogeneous objects $V \in \Vec_X^\omega$ of degree $x\in X$.

The tensor structure $\tilde{\mu}_g$ on the functor 
$T_g:\Vec_X^\omega \xrightarrow{\sim} \Vec_X^\omega, g\in G$,
is defined by
$$
\mu_g(x, y) \id_{V \ot_{\K} W} =: \tilde{\mu}_g(V, W)
:T_g(V \ot W) \xrightarrow{\sim} T_g(V) \ot T_g(W),
$$
for all homogeneous objects $V,W \in \Vec_X^\omega$ of degrees $x,y\in X$, 
respectively.

The unit-preserving structure $\nu_g$ on the functor 
$T_g:\Vec_X^\omega \xrightarrow{\sim} \Vec_X^\omega, g\in G$,
is defined by
$$
\mu^{-1}_g(e,e) \id_{\K} = : \nu_g:T_g(\K) \xrightarrow{\sim} \K.
$$

For $V,W \in \Vec$, let $\tau_{V,W}$ denote the flip operator
$V \ot_\K W \xrightarrow{\sim} W \ot _\K V: v \ot_\K w \mapsto w \ot_\K v$.
The $G$-braiding $\tilde{c}$ is defined by
$$
c(x,y) \tau_{V,W} =: \tilde{c}(V,W):
V \ot W \xrightarrow{\sim} T_g(W) \ot V,
$$
for all homogeneous objects $V,W \in \Vec_X^\omega$ of degrees
$x,y \in X$. (Here $g = \partial(x)$.)

The crossed module axioms of $(G,X,\partial)$ and the
quasi-abelian $3$-cocycle axioms of
$(\omega, \gamma, \mu, c)$ together ensure that the necessary axioms of a 
crossed category are satisfied. 
Specifically, Condition \eqref{relation between mu and omega} of Definition~\ref{quasi-abelian 3-cocycle}
ensures that $\tilde{\mu_g}$ is a tensor structure on the functor $T_g$ defined above.
Condition \eqref{relation between gamma and mu} of Definition~\ref{quasi-abelian 3-cocycle}
ensures that $\tilde{\gamma}$ is a tensor structure on the functor $T$. 
Conditions \eqref{relation between c, mu, and gamma}-\eqref{relation between c, omega, and mu}
of Definition~\ref{quasi-abelian 3-cocycle} correspond to 
Axioms \eqref{T preserves c}-\eqref{c2} of Definition~\ref{quasi-abelian 3-cocycle}, respectively. 

We will denote the crossed pointed category constructed above by
$$
\C(\omega, \gamma, \mu, c).
$$

\begin{remark}
\label{trivial component nd}
The trivial component $\C(\omega, \gamma, \mu, c)_e$
of $\C(\omega, \gamma, \mu, c)$ (under the $G$-grading) is a braided fusion category.
As a fusion category, $\C(\omega, \gamma, \mu, c)_e = \Vec_{\Ker \partial}^{\omega|_{\Ker \partial}}$.
Suppose that the quasi-abelian $3$-cocycle $(\omega, \gamma, \mu, c)$ is normalized.
Then the braiding on the trivial component is given by
$$
V \ot W \to W \ot V : v \ot w \mapsto c(x,y)w \ot v,
$$
for all homogeneous objects $V,W \in \Vec_{\Ker \partial}^{\omega|_{\Ker \partial}}$
of degrees $x,y \in \Ker \partial$. Clearly, the braided fusion category 
$\C(\omega, \gamma, \mu, c)_e$ is nondegenerate if and only if the quasi-abelian $3$-cocycle
$(\omega, \gamma, \mu, c)$ is nondegenerate in the sense of 
Definition~\ref{nondegenerate quasi-abelian 3-cocycle}. 
\end{remark}

\end{subsection}
\begin{subsection}
{Classification}

\begin{proposition}
\label{equivalence}
Let $\C(\omega, \gamma, \mu, c)$ and $\C(\omega', \gamma', \mu', c')$ be
crossed pointed categories constructed in the preceding subsection. Then
$\C(\omega, \gamma, \mu, c) \cong \C(\omega', \gamma', \mu', c')$
as crossed categories if and only if there exists an isomorphism
$(f,F)$ of the underlying (finite) crossed modules such that the quasi-abelian $3$-cocycles 
$(\omega', \gamma', \mu', c')^{(f,F)}$ and $(\omega, \gamma, \mu, c)$
are cohomologous.  
\end{proposition}
\begin{proof}
Suppose $(G,X,\partial)$ and $(G',X',\partial')$ are the underlying (finite) crossed
modules of $\C(\omega, \gamma, \mu, c)$ and $\C(\omega', \gamma', \mu', c')$,
respectively. Let $(f,F)$ be an isomorphism from 
$(G,X,\partial)$ to $(G',X',\partial')$ such that the quasi-abelian $3$-cocycles 
$(\omega', \gamma', \mu', c')^{(f,F)}$ and $(\omega, \gamma, \mu, c)$
are cohomologous via $(\eta, \beta)$ (see Section \ref{quasi-abelian cohomology}).  
In what follows we will construct an equivalence 
$(f,\tilde{F},\tilde{\eta},\eta_0,\tilde{\beta})$ of crossed
categories from $\C(\omega, \gamma, \mu, c)$ to $\C(\omega', \gamma', \mu', c')$
(see Definition \ref{crossed functor}).

Recall that as fusion categories,
$\C(\omega, \gamma, \mu, c)=\Vec_X^\omega$ and
$\C(\omega', \gamma', \mu', c') =\Vec_{X'}^{\omega'}$.
Let $V \in \Vec_X^\omega$ be a homogeneous object of degree $x \in X$.
Define a functor $\tilde{F}: \Vec_X^\omega \to \Vec_{X'}^{\omega'}$,
by $\tilde{F}(V):= V$ (as a vector space) and the degree of $\tilde{F}(V)$ 
is defined to be $F(x)$. The functor $\tilde{F}$ extends
to nonhomogeneous objects and morphisms in the obvious way.

The tensor structure $\tilde{\eta}$ on the functor $\tilde{F}$ is defined by
$$
\eta(x, y) \id_{V \ot_\K W} =: \tilde{\eta}(V, W)
:\tilde{F}(V \ot W) \xrightarrow{\sim} \tilde{F}(V) 
\ot \tilde{F}(W), 
$$
for all homogeneous objects $V,W \in \Vec_X^\omega$ of degrees $x,y\in X$, 
respectively.

The definition of the unit-preserving structure $\eta_0$ on $\tilde{F}$
is obvious. It is easy to verify that $(\tilde{F},\tilde{\eta},\eta_0)$
is an equivalence of tensor categories.

Next we define isomorphisms $\tilde{\beta}_g : F \circ T_g \xrightarrow{\sim}
T'_{f(g)} \circ F, g \in G$, of tensor functors by
$$
\beta_g(x) \id_V =: \tilde{\beta}_g(V):
(\tilde{F} \circ T_g)(V) \xrightarrow{\sim}
(T'_{f(g)} \circ \tilde{F})(V),
$$
for all homogeneous objects $V \in \Vec_X^\omega$ of degree $x\in X$.

It is easy to verify that axioms (i)-(iii) of Definition 
\ref{crossed functor} are satisfied. This
shows that $\C(\omega, \gamma, \mu, c) \cong \C(\omega', \gamma', \mu', c')$,
as crossed categories.

The converse is clear from the above construction.
\end{proof}

\begin{remark}
The above proposition, in particular, shows that if the quasi-abelian
$3$-cocycles $(\omega, \gamma, \mu, c)$ and
$(\omega', \gamma', \mu', c')$ (on the same crossed module $(G, X, \partial)$) are cohomologous, then the corresponding
crossed pointed categories $\C(\omega, \gamma, \mu, c)$ and
$\C(\omega', \gamma', \mu', c')$ are equivalent.
\end{remark}

Recall that for any crossed module $\X$ there is a natural action
of $\Aut(\X)$ on the quasi-abelian third cohomology $\H^3_{qa}(\X,\K^\times)$
of $\X$ (see Section \ref{quasi-abelian cohomology}).

\begin{theorem}
Crossed pointed categories are classified, up to equivalence,
by orbits of the quasi-abelian third cohomology $\H^3_{qa}(\X, \K^\times)$ 
(of a finite crossed module $\X$) under the action of $\Aut(\X)$.
\end{theorem}
\begin{proof}
Every crossed pointed category is equivalent to some 
$\C(\omega, \gamma, \mu, c)$ with underlying (finite) crossed module $\X$. The 
theorem now follows from Proposition \ref{equivalence}.

\end{proof}

\end{subsection}

\end{section}
\begin{section}
{Equivariantization of $\C(\omega, \gamma, \mu, c)$}
\label{equivariantization}

Throughout this section,
let $(\omega, \gamma, \mu, c)$ be a normalized quasi-abelian $3$-cocycle on a
finite crossed module $(G,X,\partial)$. In Subsection \ref{construction}
we associated to $(\omega, \gamma, \mu, c)$ a crossed
pointed category $\C(\omega, \gamma, \mu, c)$. Our goal in this section
is to apply the {\em equivariantization} process to 
$\C(\omega, \gamma, \mu, c)$ and study the resulting braided fusion category.


\begin{subsection}
{Description}
\label{description}

Recall that as a fusion category,
$\C(\omega, \gamma, \mu, c) = \Vec_X^\omega$. We begin with 
the following.

\begin{proposition}
\label{description of equivariantization}
An object of the equivariantization category $\C(\omega, \gamma, \mu, c)^G$ is a $X$-graded vector space
$V$ together with a twisted action $\rt$ of $G$ on $V$ which is compatible
with the grading, in the sense of
\begin{equation}
\label{twisted action condition}
gh \rt v = \gamma_{g,h}(x) (g \rt (h \rt v))
\end{equation}
$$
e \rt v = v, \qquad \degree(g \rt v) = \lexp{g}{x},
$$
for all $v \in V$ homogeneous of degree $x \in X$ and $g,h \in G$.
Morphisms in the category are linear maps preserving
the twisted action and grading. The twisted action of $G$ on
the tensor product is given by
\begin{equation}
\label{twisted tensor product}
g \rt(v \ot w) = \mu_g(x,y) (g \rt v \ot g \rt w),
\end{equation}
for homogeneous $v,w$ of degrees $x,y \in X$, respectively. The associativity
constraint on the category is given by
\begin{equation*}
(u \ot v) \ot w \mapsto \omega(x,y,z) u \ot (v \ot w),
\end{equation*}
for all homogeneous $u,v,w$ of degrees $x,y,z \in X$.
The braiding on the category is given by 
\begin{equation}
\label{twisted braiding}
v \ot w \mapsto c(x,y) (\partial(x) \rt w \ot v),
\end{equation}
for all homogeneous $v,w$ of degrees $x,y \in X$.
\end{proposition}
\begin{proof}
The action $\rt$ referred to in the statement of the proposition
corresponds to equivariant structure \eqref{equivariant structure}.
Equation \ref{twisted action condition} corresponds to 
\eqref{equivariant structure condition}. The definition of the tensor product 
\eqref{twisted tensor product} comes from \eqref{tensor product in equivariantization}
and the definition of the braiding \eqref{twisted braiding} comes 
\eqref{braiding in equivariantization}.
\end{proof}

\begin{remark}
There is a simple special case of the above description. 
Namely, take the quasi-abelian 
$3$-cocycle $(\omega,\gamma,\mu,c)$ (on the finite crossed
module $(G,X,\partial))$ to be trivial. Then
the corresponding equivariantization category 
$\C(1,1,1,1)^G$ admits a simple description: objects of this category are
$G$-equivariant vector bundles on $X$. We note that this braided fusion
category was considered in \cite{Ba}. This category
is not nondegenerate in general: by Proposition~\ref{when nd}
it is nondegenerate if and only if $\partial$ is an isomorphism.
In this case,  the category is equivalent to $D(G)$-$\Mod$, as a braided
fusion category. 
\end{remark}

\begin{theorem}
\label{braided g-t categories}
Every braided group-theoretical category is equivalent
to $\C(\xi)^G$, for some normalized quasi-abelian $3$-cocycle $\xi$ on
a finite crossed module $(G,X,\partial)$.
\end{theorem}
\begin{proof}
This follows from \cite{NNW}, where it was shown that every braided group-theoretical
category is the equivariantization of a pointed category.
\end{proof}

\begin{lemma}
\label{phi is a 2-cocycle}
For any $x \in X$, let $\Stab_G(x)$ denote the stabilizer of
$x$ in $G$, i.e., $\Stab_G(x) = \{g \in G \mid \lexp{g}{x}=x\}$.
Define $\phi_x:\Stab_G(x) \times \Stab_G(x) \to \K^\times$ by
$$
\phi_x(g,h) := \gamma_{g,h}(x), \qquad g,h \in \Stab_G(x).
$$
Then $\phi_x$ is a $2$-cocycle on $\Stab_G(x)$.
\end{lemma}
\begin{proof}
The Condition \eqref{condition on gamma} on $\gamma$ in Definition
\ref{quasi-abelian 3-cocycle} means that
$$
\gamma_{h,k}(x) \gamma_{g, hk}(x)
= \gamma_{gh,k}(x)  \gamma_{g,h}(\lexp{k}{x}),
$$
for all $g,h,k \in G, x \in X$. Restricting to $\Stab_G(x)$ we
get the stated assertion.
\end{proof}

Let $R$ denote a complete set of representatives of orbits of $X$
under the action of $G$.

\begin{proposition}
\label{simples}
The set of isomorphism classes of simple objects of $\C(\omega, \gamma, \mu, c)^G$
is in bijection with isomorphism classes of the set 
\begin{equation}
\Gamma := \{ (a,V) \mid a\in R, \; V \text{ is an irreducible module over } 
\K_{\phi_a}[\Stab_G(a)] \},
\end{equation}
where $\phi_a$ is the $2$-cocycle defined in Lemma \ref{phi is a 2-cocycle}.
\end{proposition}
\begin{proof}
Let $\Irr(\C(\omega, \gamma, \mu, c)^G)$ denote the set of simple objects 
of $\C^G$. We will define a map
\begin{equation}
\label{induction map}
\Gamma \to \Irr(\C(\omega, \gamma, \mu, c)^G)
\end{equation}
and show that it induces a bijection between the isomorphism classes of the
source and target sets.  Let $g_1,g_2, \cdots$ 
be coset representatives of $\Stab_G(a)$ in $G$. 
Pick any $(a,V) \in \Gamma$.
We define the map \eqref{induction map} by
\begin{equation}
\label{induction}
(a,V) \mapsto 
\tilde{V} = \oplus_{g_i} V_{\lexp{g_i}{a}}, 
\end{equation}
where $V_{\lexp{g_i}{a}} = V$ as a vector space and
$\degree(V_{\lexp{g_i}{a}}) = \lexp{g_i}{a}$. The twisted action of
$G$ on $\tilde{V}$ is given by
\begin{equation}
\label{induction action}
h \rt v := \frac{\gamma_{g_j,t}(a)}{\gamma_{h, g_i}(a)} (t \rt v),
\end{equation}
for all $v \in \tilde{V}$ homogeneous of degree $\lexp{g_i}{a}$ with
$t \in \Stab_G(a)$ uniquely determined by the equation $hg_i = g_jt$.
The degree of $h \rt v$ is defined to be $\lexp{g_j}{a}$.

To prove that the map \eqref{induction map} (defined via \eqref{induction}
and \eqref{induction action}) is well-defined we need to show that
the action defined in \eqref{induction action} satisfies \eqref{twisted action condition}.
This amounts to verifying that the scalar
$$
\frac{\gamma_{g_k,st}(a) \gamma_{s,t}(a)}{\gamma_{gh,g_i}(a)}
$$
is equal to the scalar
$$
\frac{\gamma_{g,h}(\lexp{g_i}{a})\gamma_{g_j,t}(a) \gamma_{g_k,s}(a)}
{\gamma_{h,g_i}(a) \gamma_{g,g_j}(a)}
$$
for all $g,h \in G, s,t \in \Stab_G(a)$ with $hg_i=g_jt$ and $gg_j=g_ks$.
The equality of the two scalars follows from applying Condition \eqref{condition on gamma}
on $\gamma$ in Definition \ref{quasi-abelian 3-cocycle} successively to the
quadruples $(g,h,g_i,a),(g,g_j,t,a),(g_k,s,t,a)$.

We now show that the map \eqref{induction map} induces a bijection between
isomorphism classes of the source and target sets.
It is clear that the map \eqref{induction map} preserves isomorphic objects.
Furthermore, the object in $\Irr(\C(\omega, \gamma, \mu, c)^G)$
corresponding to $(a,V) \in \Gamma$ has FP-dimension
equal to $|\lexp{G}{a}| \dim_\K V$, where $\lexp{G}{a}$ denotes the orbit containing $a$.
The sum of squares of FP-dimensions of isomorphism classes of objects in the image
of \eqref{induction map} is 
\begin{equation}
\begin{split}
\sum_{a \in R} \;\; \sum_{V \in \Irr(\K_{\phi_a [\Stab_G(a)]})} |\lexp{G}{a}|^2 (\dim V)^2 
&= \sum_{a \in R} |\lexp{G}{a}|^2 |\Stab_G(a)| \\
&= \sum_{a \in R} |\lexp{G}{a}||G| \\
&= |G||X| \\
&= \FPdim(\C(\omega, \gamma, \mu, c)^G),
\end{split}
\end{equation}
completing the proof.
\end{proof}

\end{subsection}

\begin{subsection}
{Twist and $S$-matrix}
\label{twist and S-matrix}
As before, $R$ denotes a complete set of representatives of orbits of $X$
under the action of $G$.
By Proposition~\ref{simples}, the simple objects of $\C(\omega,\gamma,\mu,c)^G$
correspond to pairs $(a,\chi)$, where $a \in R$ and $\chi$ is an irreducible
$\phi_a$-character of $\Stab_G(a)$. 
Note that $\C(\omega,\gamma,\mu,c)^G$ admits a canonical twist $\theta$
with respect to which categorical dimensions coincide with FP-dimensions.
The values of $\theta$ on simple objects are given by
$$
\theta_{(a,\chi)} = c(a,a) \frac{\chi(\partial(a))}{\deg \chi}.
$$
A direct calculation shows that the $S$-matrix $S$ is given by
$$
S_{(a,\chi), (b,\chi')} = 
\sum_{\substack{x \in (\lexp{G}{a})\\ y \in (\lexp{G}{b}) \cap C_X(x)}}
c(x,y)c(y,x) \frac{\gamma_{g,\partial(\lexp{g^{-1}}{y})}(a) 
\gamma_{h,\partial(\lexp{h^{-1}}{x})}(b)}
{\gamma_{\partial(y),g}(a) \gamma_{\partial(x),h}(b)}
\chi(\lexp{g^{-1}}{y}) \chi'(\lexp{h^{-1}}{x}),
$$
where in each summand $g$ and $h$ are defined by 
$\lexp{g}{a}=x$ and $\lexp{h}{b}=y$. (Note that the choice
of $g$ and $h$ does not affect the sum.)

\end{subsection}
\begin{subsection}
{Modularity}

As before, $(\omega,\gamma,\mu,c)$
is a normalized quasi-abelian $3$-cocycle on a finite crossed
module $(G,X,\partial)$. We have the following. 

\begin{proposition}
\label{when nd}
The braided fusion category $\C(\omega,\gamma,\mu,c)^G$ is nondegenerate
if and only if the homomorphism $\partial$ is surjective and 
$(\omega,\gamma,\mu,c)$ is nondegenerate in the sense of
Definition~\ref{nondegenerate quasi-abelian 3-cocycle}.
\end{proposition}
\begin{proof}
This follows immediately by combining Remark~\ref{when C^G is nd} and
Remark~\ref{trivial component nd}.
\end{proof}

Assume that $\partial$ is surjective and 
$(\omega,\gamma,\mu,c)$ is nondegenerate. Then
$\C(\omega,\gamma,\mu,c)^G$ together with the 
canonical twist given in Subsection~\ref{twist and S-matrix}
is a modular category, i.e., the $S$-matrix described in
Subsection~\ref{twist and S-matrix} is invertible.  In this situation,
using the orthogonality relations for projective characters we obtain that
the Gauss sum and central charge, respectively, of $\C(\omega,\gamma,\mu,c)^G$ are given
by
\begin{eqnarray*}
\tau(\C(\omega,\gamma,\mu,c)^G) &=& |G| \sum_{a \in \Ker \partial}c(a,a),\\
\zeta(\C(\omega,\gamma,\mu,c)^G) &=& \frac{1}{\sqrt{|\Ker \partial|}} \sum_{a \in \Ker \partial}c(a,a).
\end{eqnarray*}

\begin{note}
The sum $\sum_{a \in \Ker \partial}c(a,a)$ is the classical Gauss sum for the
quadratic form $a \mapsto c(a,a)$ on the abelian group $\Ker \partial$.
\end{note}

\begin{remark}
Note that the category $\C(\omega,\gamma,\mu,c)^G$ may admit other twists
(besides the canonical one).
In view of Theorem~\ref{braided g-t categories} and Proposition~\ref{when nd}, 
a description of all twists on $\C(\omega,\gamma,\mu,c)^G$ will imply a 
description of {\em all} modular group-theoretical categories. The former is
easily obtained using Proposition~\ref{twists and self-dual invertibles}.
\end{remark}
\end{subsection}

\end{section}
\begin{section}
{Quasi-triangular quasi-hopf algebra arising from quasi-abelian $3$-cocycles
on crossed modules}

Let $(\omega, \gamma, \mu, c)$ be a normalized quasi-abelian $3$-cocycle
on a finite crossed module $(G,X,\partial)$. 
In the previous section we described the braided fusion category
$\C(\omega, \gamma, \mu, c)^G$. This category is integral (i.e., the FP-dimensions of
objects are integers), so
there exists a finite-dimensional quasi-triangular quasi-Hopf algebra $H$ such that
$\C(\omega, \gamma, \mu, c)^G \cong H$-$\Mod$, as braided fusion categories (see \cite[Theorem~8.33]{ENO} and
\cite[Section XV.2]{K}).
Our goal in this section is to describe such an $H$.

In what follows we associate to $(\omega, \gamma, \mu, c)$
a finite-dimensional quasi-triangular quasi-Hopf algebra 
$H(\omega, \gamma, \mu, c)$, which may be viewed as a generalization
of the twisted Drinfeld double of a finite group. Let $H(\omega, \gamma, \mu, c)$
be a finite-dimensional vector space with a basis $\{t_xg\}_{(x,g) \in X \times G}$
indexed by the set $X \times G$.
Define a product on $H(\omega, \gamma, \mu, c)$ by
\begin{equation}
\label{product}
(t_xg)(t_yh) := \delta_{x, \lexp{h}{y}} \gamma_{g,h}(y)^{-1} t_y(gh).
\end{equation}
This product admits a unit
\begin{equation}
\label{unit}
1 = \sum_{x \in X} t_xe.
\end{equation}
Define a coproduct $\Delta: H(\omega, \gamma, \mu, c)
\to H(\omega, \gamma, \mu, c) \ot H(\omega, \gamma, \mu, c)$
and counit $\varepsilon: H(\omega, \gamma, \mu, c) \to \K$ by
\begin{equation}
\label{coproduct}
\Delta(t_xg) := \sum_{a,b \in X: ab=x} \mu_g(a,b) t_ag \ot t_bg
\end{equation}
and
\begin{equation}
\label{counit}
\varepsilon(t_xg) := \delta_{x,e}.
\end{equation}
Also, set
\begin{equation}
\label{Drinfeld associator}
\Phi := \sum_{x,y,z \in X} \omega(x,y,z) t_xe \ot t_ye \ot t_ze,
\end{equation}
\begin{equation}
\label{R-matrix}
R := \sum_{x,y \in X} c(x,y) t_xe \ot t_y \partial(x),
\end{equation}
and
\begin{equation}
\label{alpha beta}
\alpha := 1, \qquad \beta := \sum_{x \in X} \omega(x^{-1},x,x^{-1}) t_xe.
\end{equation}
Finally, define a linear map $S:H(\omega, \gamma, \mu, c) \to
H(\omega, \gamma, \mu, c)$ by
\begin{equation}
\label{antipode}
S(t_xg) := \frac{\gamma_{g^{-1},g}(x^{-1})}{\mu_g(x,x^{-1})} t_{\lexp{g}{x^{-1}}}g^{-1}.
\end{equation}

\begin{proposition}
The product \eqref{product},
unit \eqref{unit}, coproduct $\Delta$ \eqref{coproduct}, counit $\varepsilon$ \eqref{counit}, 
Drinfeld associator $\Phi$ \eqref{Drinfeld associator}, and anti-automorphism $S$
\eqref{antipode} make $H(\omega, \gamma, \mu, c)$ a quasi-triangular quasi-Hopf algebra 
with universal $R$-matrix $R$
\eqref{R-matrix} in the sense of \cite[Definitions 1.1, 2.1, and 5.1]{K}.
\end{proposition}
\begin{proof}
The proof is completely similar to the one for the twisted Drinfeld double of 
a finite group:
Associativity of the product is equivalent to the equality
$$
\gamma_{h,k}(x) \gamma_{g, hk}(x)
= \gamma_{gh,k}(x)  \gamma_{g,h}(\lexp{k}{x}),
\qquad g,h,k \in G, x \in X,
$$
which holds by Axiom \eqref{condition on gamma} in Definition
\ref{quasi-abelian 3-cocycle}. 
Quasi-coassociativity of the coproduct 
is equivalent to the equality
$$
\frac{\mu_g(y,z) \mu_g(x,yz)}{\mu_g(xy,z) \mu_g(x,y)} = 
\frac{\omega(\lexp{g}{x},\lexp{g}{y},\lexp{g}{z})}{\omega(x,y,z)},
\qquad g \in G, x,y,z \in X,
$$
which holds by Axiom \eqref{relation between mu and omega} in Definition
\ref{quasi-abelian 3-cocycle}. 
That the coproduct is a morphism of algebras
is equivalent to the equality
$$
\frac{\gamma_{g,h}(x) \gamma_{g,h}(y)}{\gamma_{g,h}(xy)} 
= \frac{\mu_g(\lexp{h}{x}, \lexp{h}{y}) \mu_h(x,y)}{\mu_{gh}(x,y)},
\qquad g,h \in G, x,y \in X,
$$
which holds by Axiom \eqref{relation between gamma and mu} in Definition
\ref{quasi-abelian 3-cocycle}.

We note that the inverse of the $R$-matrix $R$ is
$$
R^{-1} = \sum_{x,y \in X} c(x,x^{-1}yx)^{-1} 
\gamma_{\partial(x),\partial(x^{-1})}(y)^{-1} t_xe \ot t_y\partial(x^{-1}).
$$
The $R$-matrix axioms on $R$ hold due to Axioms (e)-(g) in Definition
\ref{quasi-abelian 3-cocycle}.

Finally, Axioms (a)-(d) in Definition
\ref{quasi-abelian 3-cocycle} ensure that $S$ is indeed an 
anti-automorphism that satisfies the required axioms.
\end{proof}

\begin{proposition}
Let $(\omega, \gamma, \mu, c)$ be a normalized quasi-abelian $3$-cocycle
on a finite crossed module $(G,X,\partial)$. The categories 
$\C(\omega, \gamma, \mu, c)^G$ (see Section \ref{equivariantization})
and $H(\omega, \gamma, \mu, c)$-$\Mod$ are equivalent as braided
fusion categories. 
\end{proposition}
\begin{proof}
Let $V$ be a (left) module over $H(\omega, \gamma, \mu, c)$, with
action denoted by ``$\cdot$". Note that
$V$ admits an $X$-grading: $V = \oplus_{x \in X} V_x$, where
$V_x = (t_xe) \cdot V$. Define a twisted action of $G$ on $V$ by
$$
g \rt v := (t_xg) \cdot v,
$$
for all $v \in V$ homogeneous of degree $x \in X$. Observe that the
degree of $g \rt v$ is $\lexp{g}{x}$, as 
$(t_xg)(t_xe) = (t_{\lexp{g}{x}}e)(t_xg)$. The aforementioned action
is twisted in the sense that 
$$
gh \rt v = \gamma_{g,h}(x) (g \rt (h \rt v)).
$$
Note that the twisted action of $G$ completely determines the action
of $H(\omega, \gamma, \mu, c)$ on the module $V$.
The associativity constraint on the category $H(\omega, \gamma, \mu, c)$-$\Mod$
(which is defined using the Drinfeld associator $\Phi$ \eqref{Drinfeld associator}) 
is given by
$$
(u \ot v) \ot w \mapsto \omega(x,y,z) u \ot (v \ot w),
$$
for all homogeneous $u,v,w$ of degrees $x,y,z \in X$.
The braiding on the category $H(\omega, \gamma, \mu, c)$-$\Mod$ (which is defined
using the $R$-matrix $R$ \eqref{R-matrix}) is given by
$$
v \ot w \mapsto c(x,y) (\partial(x) \rt w \ot v),
$$
for homogeneous $v,w$ of degrees $x,y \in X$.
Comparing with Proposition~\ref{description of equivariantization} 
we obtain the stated assertion.
\end{proof}

We next explain the relation between the quasi-triangular quasi-Hopf
algebras constructed above and the twisted Drinfeld double of a 
finite group. Let $\omega$ be a normalized $3$-cocycle on a finite group $G$.

Define
\begin{equation}
\label{gamma omega}
\gamma_{g,h}(x) := \frac{\omega(g,h,x)\omega(ghxh^{-1}g^{-1},g,h)}
{\omega(g, hxh^{-1},h)}
\end{equation}
and
\begin{equation}
\label{mu omega}
\mu_g(x,y) := \frac{\omega(gxg^{-1},g,y)}
{\omega(gxg^{-1},gyg^{-1},g) \omega(g,x,y)},
\end{equation}
for all $g,h,x,y \in G$.

A direct computation establishes the following.

\begin{lemma}
\label{3-cocycle to quasi-abelian 3-cocycle}
The quadruple $(\omega, \gamma, \mu, 1)$ where
$\gamma$ and $\mu$ are defined by \eqref{gamma omega}
and \eqref{mu omega}, respectively, is a quasi-abelian
$3$-cocycle on the crossed module $(G,G, \id_G)$
(where $G$ acts on itself by conjugation) 
in the sense of Definition \ref{quasi-abelian 3-cocycle}.
\end{lemma}

Let $(\omega, \gamma, \mu, 1)$ be the quasi-abelian
$3$-cocycle on $(G,G,\id_G)$ constructed in Lemma \ref{3-cocycle to quasi-abelian 3-cocycle}. 
Then, evidently, $H(\omega^{-1}, \gamma^{-1}, \mu^{-1}, 1) \cong D^{\omega}(G)$,
as quasi-triangular quasi-Hopf algebras.
In particular, $\C(\omega^{-1}, \gamma^{-1}, \mu^{-1}, 1)^G \cong D^\omega(G)$-$\Mod$, as 
braided fusion categories.

\end{section}

\bibliographystyle{ams-alpha}


\end{document}